\documentclass[final,reqno]{siamltex}
\usepackage{amsmath,amssymb,bm,upgreek,cite,nicefrac}
\usepackage[mathscr]{eucal}

\DeclareMathOperator{\Int}{int}
\newcommand{\cX}{\mathcal{X}}
\newcommand{\cY}{\mathcal{Y}}
\newcommand{\cH}{\mathcal{H}}
\newcommand{\cG}{\mathcal{G}}
\newcommand{\cS}{\mathcal{S}}
\DeclareMathOperator{\Exp}{\mathbb{E}}
\DeclareMathOperator{\Prob}{\mathbb{P}}
\newcommand{\D}{\mathrm{d}}
\newcommand{\E}{\mathrm{e}}
\newcommand{\df}{\triangleq}
\newcommand{\transp}{^{\mathsf{T}}}
\newcommand{\Uadm}{\mathfrak{U}}
\newcommand{\Usr}{\mathfrak{U}_{\mathrm{SM}}}
\newcommand{\Ussr}{\mathfrak{U}_{\mathrm{SSM}}}
\newcommand{\Lg}{L}
\newcommand{\sF}{\mathfrak{F}}
\newcommand{\Ind}{\mathbb{I}}
\newcommand{\Cc}{\mathcal{C}}
\newcommand{\Lp}{\mathcal{L}}
\newcommand{\Sob}{\mathscr{W}}
\newcommand{\Sobl}{\mathscr{W}_{\mathrm{loc}}}
\newcommand{\Act}{\mathbb{U}}
\newcommand{\imeas}{\mu}
\newcommand{\Lyap}{\mathcal{V}}
\newcommand{\RR}{\mathbb{R}}
\newcommand{\NN}{\mathbb{N}}
\newcommand{\abs}[1]{\lvert#1\rvert}
\newcommand{\babs}[1]{\bigl\lvert#1\bigr\rvert}

\newcommand{\norm}[1]{\lVert#1\rVert}
\newcommand{\bnorm}[1]{\bigl\lVert#1\bigr\rVert}
\newtheorem{remark}[theorem]{{\it Remark}}
\newtheorem{assumption}[theorem]{{\it Assumption}}

\title
{A Relative Value Iteration Algorithm for\\ Non-degenerate
Controlled Diffusions}
\author{Ari Arapostathis\thanks{Department of Electrical and Computer
Engineering, The University of Texas at Austin, 1 University Station,
Austin, TX 78712 (ari@mail.utexas.edu).
This author's  work was supported in part by the Office of Naval Research
through the Electric Ship Research and Development Consortium.}
\and Vivek S. Borkar\thanks{Department of Electrical Engineering,
Indian Institute of Technology, Powai, Mumbai 400076, India.
(borkar.vs@gmail.com). This author's work was supported in part by the 
J.~C.~Bose Fellowship from the Government of India.}}

\begin{document}
\maketitle

\begin{abstract}
The ergodic control problem for a non-degenerate controlled diffusion
controlled through its drift is considered under a uniform stability
condition that ensures the well-posedness of the associated
Hamilton--Jacobi--Bellman (HJB) equation.
A nonlinear parabolic evolution equation is then proposed as a
continuous time continuous state space analog of White's
`\textit{relative value iteration}' algorithm for solving the ergodic
dynamic programming equation for the finite state finite action case.
Its convergence to the solution of the HJB equation is established using
the theory of monotone dynamical systems and also, alternatively,
by using the theory of reverse martingales.
\end{abstract}



\pagestyle{myheadings}
\thispagestyle{plain}
\markboth{ARI ARAPOSTATHIS AND VIVEK S. BORKAR}
{RELATIVE VALUE ITERATION FOR CONTROLLED DIFFUSIONS}

\section{Introduction}\label{S-intro}

Consider a controlled Markov chain on a finite state space
$\cS=\{1,\dotsc,N\}$
with transition probabilities $p_{ij}(u)$, $i,j\in\cS$,
which depend continuously on a control parameter $u$ that lives in
a compact `action' space $\Act$,  such
that when in state $i$ the control $u$ is chosen from a compact
subset $\Act_{i}\subset\Act$.
Assuming irreducibility for the stochastic matrix
$P^{v}\df \left[p_{i,j}(v_{i})\right]_{i,j\in\cS}$
with $v=(v_{1},\dotsc,v_{N})\in (\Act_{1}\times\dotsb\times\Act_{N})$,
consider the control problem
of minimizing the average (or ergodic) cost
\begin{equation*}
\limsup_{n\uparrow\infty}\;\frac{1}{n}\;\sum_{k=0}^{n-1}
\Exp\left[r(X_{k},U_{k})\right]
\end{equation*}
for a prescribed $r:\cS\times\Act\to\RR$ and control
sequence $\{U_{k}\}$ such that $U_{k}\in \Act_{X_{k}}$ and
\begin{equation*}
\Prob(X_{n+1} = j \mid X_{m}\,,U_{m}\,,~ m\le n)
= p_{X_{n}\,j} (U_{n})\,,\quad n\ge0\,.
\end{equation*}
The dynamic programming equation for this problem is the well known
controlled Poisson equation:
\begin{equation*}
V(i) = \min_{u\in\Act_{i}}\;\left[r(i,u) - \beta + \sum_{j\in\cS}
p_{ij}(u) V(j)\right]\,,\quad i\in\cS\,.
\end{equation*}
This is an equation in unknowns $(V,\beta)$, with
$V=\bigl(V(1),\dotsc,V(N)\bigr)\in\RR^{N}$
the so called value function.
Under the irreducibility hypothesis above, $V$ is uniquely specified
modulo an additive constant and $\beta$ is uniquely specified as the
optimal ergodic cost.
See \cite{DynYus-79,Puterman} for details.

By analogy with the value iteration algorithm for the discounted cost
problem, one may consider the value iteration algorithm
\begin{equation}\label{E-1.1}
V^{n+1}(i)  = \min_{u\in\Act_{i}}\; \left[ r(i,u) -\beta + \sum_{j\in\cS}
p_{ij}(u) V^{n}(j)\right]\,,\quad i\in\cS\,,
\end{equation}
beginning with an initial guess $V^{0}(\cdot)$.
The difficulty here is that $\beta$ is unknown as well.
On the other hand, if we drop $\beta$ from \eqref{E-1.1}, there is
no convergence---the map $V^{n}\mapsto V^{n+1} \df F(V^{n})$ that is
being iterated lacks the contractivity property of its discounted cost
counterpart.
Thus clearly some renormalization is required.
The earliest example of such a \emph{relative value iteration} algorithm for
finite state Markov chains is perhaps that of White \cite{White-63},
which is governed by
\begin{subequations}
\begin{align}
h_{k+1}(i) &= \min_{u\in\Act_{i}}\;
\Biggl[r(i,u) +\sum_{j=1}^{n} p_{ij}(u) h_{k}(j)\Biggr]
- \lambda_{k+1}\label{E-Whitea}\\[5pt]
\lambda_{k+1} &= \min_{u\in\Act_{n}}\;
\left[ r(n,u) +\sum_{j=1}^{n} p_{nj}(u) h_{k}(j)\right]\,.\label{E-Whiteb}
\end{align}
\end{subequations}
For a discussion of other possible choices for updating \eqref{E-Whiteb}
see \cite{Abounadi-01}.

Bertsekas introduced in \cite{Bertsekas-98} a variation of this method
that takes the form
\begin{align*}
h_{k+1}(i) &= \min_{u\in\Act_{i}}\;
\Biggl[r(i,u) +\sum_{j=1}^{n-1} p_{ij}(u) h_{k}(j)\Biggr]
- \lambda_{k}\\[5pt]
\lambda_{k+1} &= \lambda_{k}+ \gamma_{k} h_{k+1}(n)\,.
\end{align*}
Here $\{\gamma_{k}\}$ is a sequence of positive stepsizes.
This has led to the learning algorithms analyzed in \cite{Abounadi-01}.
Recently Shlakhter et.\ al.\ \cite{Shlakhter-10} have studied ways
of accelerating the convergence of the above value iteration algorithms.

Studies of convergence of relative value iteration schemes for
more general Markov processes are non-existent.
The only related work that comes to mind is convergence
of the value iteration in \eqref{E-1.1}
for denumerable controlled Markov chains \cite{Aviv-99}.

Our aim in this paper is to propose a relative value iteration scheme
in continuous time and space for a class of controlled diffusion processes
and prove its convergence.
While we prefer to think of this scheme as a continuous time and space
relative value iteration, it can also be viewed as a `stabilization
of a nonlinear parabolic PDE problem in the sense of Has$'$minski\u{\i}
(see \cite{Hasm-60}).
We follow two different approaches for the proof of convergence, based on
resp.\ the theory of monotone dynamical systems and the theory
of reverse martingales.
These should be of independent interest.

The paper is organized as follows.
The next section describes the ergodic control problem for diffusions
and the associated Hamilton--Jacobi--Bellman equation, leading to the proposed
relative value iteration scheme.
Section~\ref{S3} provides a motivating illustration from the discrete state
counterpart, introduces some notation, and recalls some key results from
parabolic PDEs and monotone dynamical systems for later use.
Section~\ref{S-main} gives the two convergence proofs alluded in the Abstract,
while Section~\ref{S-concl} concludes with some pointers to future work.

\section{Problem statement}\label{S2}

\subsection{The model}

We are concerned with
controlled diffusion processes $X = \{X_{t},\;t\ge0\}$
taking values in the $d$-dimensional Euclidean space $\RR^{d}$, and
governed by the It\^o stochastic differential equation
\begin{equation}\label{e-sde}
\D{X}_{t} = b(X_{t},U_{t})\,\D{t} + \upsigma(X_{t})\,\D{W}_{t}\,.
\end{equation}
All random processes in \eqref{e-sde} live in a complete
probability space $(\Omega, \sF,\Prob)$.
The process $W$ is a $d$-dimensional standard Wiener process independent
of the initial condition $X_{0}$.
The control process $U$ takes values in a compact, metrizable set $\Act$, and
$U_{t}(\omega)$ is jointly measurable in
$(t, \omega)\in[0,\infty)\times\Omega$.
Moreover, it is \emph{non-anticipative}:
for $s < t$, $W_{t} - W_{s}$ is independent of
\begin{equation*}
\sF_{s} \df \text{the completion of~} \sigma\{X_{0}, U_{r}, W_{r},\; r\le s\}
\text{~relative to~} (\sF,\Prob)\,.
\end{equation*}
Such a process $U$ is called an \emph{admissible control},
and we let $\Uadm$ denote the set of all admissible controls.

We impose the following standard assumptions on the drift $b$
and the diffusion matrix $\upsigma$
to guarantee existence and uniqueness of solutions to \eqref{e-sde}.
\begin{description}
\item[(A1)]
\textit{Local Lipschitz continuity:\/}
The functions
\begin{equation*}
b=\bigl[b^{1},\dotsc,b^{d}\bigr]\transp :\RR^{d} \times\Act\mapsto\RR^{d}
\quad\text{and}\quad
\upsigma=\bigl[\upsigma^{ij}\bigr]:\RR^{d}\mapsto\RR^{d\times d}
\end{equation*}
are locally Lipschitz in $x$ with a Lipschitz constant $\kappa_{R}$ depending on
$R>0$.
In other words, if $B_{R}$ denotes the open ball of radius $R$ centered
at the origin in $\RR^{d}$, then
for all $x,y\in B_{R}$ and $u\in\Act$,
\begin{equation*}
\abs{b(x,u) - b(y,u)} + \norm{\upsigma(x) - \upsigma(y)}
\le \kappa_{R}\abs{x-y}\,,
\end{equation*}
where $\norm{\upsigma}^{2}\df
\mathrm{trace}\left(\upsigma\upsigma\transp\right)$.
\item[(A2)]
\textit{Affine growth condition:\/}
$b$ and $\upsigma$ satisfy a global growth condition of the form
\begin{equation*}
\abs{b(x,u)}^{2}+ \norm{\upsigma(x)}^{2}\le \kappa_{1}
\bigl(1 + \abs{x}^{2}\bigr) \qquad \forall (x,u)\in\RR^{d}\times\Act\,.
\end{equation*}
\item[(A3)]
\textit{Local non-degeneracy:\/}
Let $a \df \frac{1}{2}\upsigma\,\upsigma\transp$.
For each $R>0$, we have
\begin{equation*}
\sum_{i,j=1}^{d} a^{ij}(x)\xi_{i}\xi_{j}
\ge\kappa^{-1}_{R}\abs{\xi}^{2} \qquad\forall x\in B_{R}\,,
\end{equation*}
for all $\xi=(\xi_{1},\dotsc,\xi_{d})\in\RR^{d}$.
\end{description}
We also assume that $b$ is continuous in $(x,u)$.

In integral form, \eqref{e-sde} is written as
\begin{equation}\label{E2}
X_{t} = X_{0} + \int_{0}^{t} b(X_{s},U_{s})\,\D{s}
+ \int_{0}^{t} \upsigma(X_{s})\,\D{W}_{s}\,.
\end{equation}
The second term on the right hand side of \eqref{E2} is an It\^o
stochastic integral.
We say that a process $X=\{X_{t}(\omega)\}$ is a solution of \eqref{e-sde},
if it is $\sF_{t}$-adapted, continuous in $t$, defined for all
$\omega\in\Omega$ and $t\in[0,\infty)$, and satisfies \eqref{E2} for
all $t\in[0,\infty)$ at once a.s.

With $u\in\Act$ treated as a parameter, we define
the family of operators $\Lg^{u}:\Cc^{2}(\RR^{d})\mapsto\Cc(\RR^{d})$ by
\begin{equation}\label{E3}
\Lg^{u} f(x) = \sum_{i,j} a^{ij}(x)
\frac{\partial^{2}f}{\partial x_{i} \partial x_{j}} (x)
+\sum_{i} b^{i} (x,u) \frac{\partial f}{\partial x_{i}} (x)\,,\quad u\in\Act\,.
\end{equation}
We refer to $\Lg^{u}$ as the \emph{controlled extended generator} of
the diffusion.

Of fundamental importance in the study of functionals of $X$ is
It\^o's formula.
For $f\in\Cc^{2}(\RR^{d})$ and with $\Lg^{u}$ as defined in \eqref{E3},
\begin{equation}\label{E4}
f(X_{t}) = f(X_{0}) + \int_{0}^{t}\Lg^{U_{s}} f(X_{s})\,\D{s}
+ M_{t}\,,\quad\text{a.s.},
\end{equation}
where
\begin{equation*}
M_{t} \df \int_{0}^{t}\bigl\langle\nabla f(X_{s}),
\upsigma(X_{s})\,\D{W}_{s}\bigr\rangle
\end{equation*}
is a local martingale.
Krylov's extension of the It\^o formula \cite[p.~122]{Krylov}
extends \eqref{E4} to functions $f$ in the local
Sobolev space $\Sobl^{2,p}(\RR^{d})$.

Recall that a control is called \emph{Markov}  if
$U_{t} = v(t,X_{t})$ for a measurable map $v :\RR\times\RR^{d}\mapsto \Act$,
and it is called \emph{stationary Markov} if $v$ does not depend on
$t$, i.e., $v :\RR^{d}\mapsto \Act$.
Correspondingly, the equation 
\begin{equation}\label{E5}
X_{t} = x_{0} + \int_{0}^{t} b\bigl(X_{s},v(s,X_{s})\bigr)\,\D{s} +
\int_{0}^{t} \upsigma(X_{s})\,\D{W}_{s}
\end{equation}
is said to have a \emph{strong solution}
if given a Wiener process $(W_{t},\sF_{t})$
on a complete probability space $(\Omega,\sF,\Prob)$, there
exists a process $X$ on $(\Omega,\sF,\Prob)$, with $X_{0}=x_{0}\in\RR^{d}$,
which is continuous,
$\sF_{t}$-adapted, and satisfies \eqref{E5} for all $t$ at once, a.s.
A strong solution is called \emph{unique},
if any two such solutions $X$ and $X'$ agree
$\Prob$-a.s., when viewed as elements of $\Cc\bigl([0,\infty),\RR^{d}\bigr)$.
It is well known that under Assumptions (A1)--(A3),
for any Markov control $v$,
\eqref{E5} has a unique strong solution \cite{Gyongy-96}.

Let $\Usr$ denote the set of stationary Markov controls.
Under $v\in\Usr$, the process $X$ is strong Markov,
and we denote its transition function by $P^{t}_{v}(x,\cdot)$.
It also follows from the work of \cite{Bogachev-01,Stannat-99} that under
$v\in\Usr$, the transition probabilities of $X$
have densities which are locally H\"older continuous.
Thus $\Lg^{v}$ defined by
\begin{equation*}
\Lg^{v} f(x) = \sum_{i,j} a^{ij}(x)
\frac{\partial^{2}f}{\partial x_{i} \partial x_{j}} (x)
+\sum_{i} b^{i} (x,v(x)) \frac{\partial f}{\partial x_{i}} (x)\,,\quad v\in\Usr\,,
\end{equation*}
for $f\in\Cc^{2}(\RR^{d})$,
is the generator of a strongly-continuous
semigroup on $\Cc_{b}(\RR^{d})$, which is strong Feller.
We let $\Prob_{x}^{v}$ denote the probability measure and
$\Exp_{x}^{v}$ the expectation operator on the canonical space of the
process under the control $v\in\Usr$, conditioned on the
process $X$ starting from $x\in\RR^{d}$ at $t=0$.

\subsection{The ergodic control problem}

Let $r\colon \RR^{d} \times\Act\to\RR$ be a continuous function
bounded from below, referred to as the \emph{running cost}.
As is well known, the ergodic control problem, in its \emph{almost sure}
(or \emph{pathwise}) formulation,
seeks to a.s.\ minimize over all admissible $U\in\Uadm$
\begin{equation}\label{E-ergcrit}
\limsup_{t\to\infty}\; \frac{1}{t}\int_{0}^{t} r(X_{s},U_{s})\,\D{s}\,.
\end{equation}
A weaker, \emph{average} formulation seeks to minimize
\begin{equation}\label{E-avgcrit}
\limsup_{t\to\infty}\;\frac{1}{t}\int_{0}^{t}
\Exp^{U}\bigl[r(X_{s},U_{s})\bigr]\,\D{s}\,.
\end{equation}
We let $\beta$ be defined as 
\begin{equation}\label{E-beta}
\beta\df \inf_{U\in\Uadm}\; \limsup_{t\to\infty}\;\frac{1}{t}\int_{0}^{t}
\Exp^{U}\bigl[r(X_{s},U_{s})\bigr]\,\D{s}\,,
\end{equation}
i.e., the infimum of \eqref{E-avgcrit} over all
admissible controls.

We assume that the running
cost function $r\colon\RR^{d}\times\Act\to\RR_{+}$
is continuous and locally Lipschitz
in its first argument uniformly in $u\in\Act$.
Without loss of generality we let $\kappa_{R}$
be a Lipschitz constant of $r$ over $B_{R}$, i.e.,
More specifically, for some function $K_{c}\colon\RR_{+}\to\RR_{+}$,
\begin{equation*}
\babs{r(x,u)-r(y,u)} \le \kappa_{R}\abs{x-y}
\qquad\forall x,y\in B_{R}\,,~\forall u\in\Act\,,
\end{equation*}
and all $R>0$.

We work under the following stability assumption:
\begin{assumption}\label{A01}
There exists a nonnegative, inf-compact $\Lyap:\RR^{d}\to \RR$
and positive constants $c_{0}$, $c_{1}$ and $c_{2}$ satisfying
\begin{subequations}
\begin{align}
\Lg^{u} \Lyap(x) &\le  c_{0} - c_{1} \Lyap(x) \qquad \forall u\in\Act
\label{e-lyap}\\[5pt]
\sup_{u\in\Act}\;r(x,u)&\le c_{2} \Lyap(x)
\end{align}
\end{subequations}
for all $x\in\RR^{d}$.
Without loss of generality we assume $\Lyap\ge1$.
\end{assumption}

It is well known (see \cite{book,GihSko72}) that \eqref{e-lyap} implies that
\begin{equation}\label{e-estlyap}
\Exp^{U}_{x} \left[\Lyap(X_{t}) \right] \le
\frac{c_{0}}{c_{1}} + \Lyap(x) \E^{-c_{1} t}\qquad \forall x\in\RR^{d}\,,
\quad\forall U\in\Uadm\,.
\end{equation}

Recall that control $v\in\Usr$ is called \emph{stable}
if the associated diffusion is positive recurrent.
We denote the set of such controls by $\Ussr$.
Also we let $\imeas_{v}$ denote the unique invariant probability
measure on $\RR^{d}$ for the diffusion under the control $v\in\Ussr$.
It follows by \eqref{e-estlyap} that, under Assumption~\ref{A01},
all stationary Markov controls are stable and that
\begin{equation*}
\int_{\RR^{d}} \Lyap(x)\,\imeas_{v}(\D{x})\le
\frac{c_{0}}{c_{1}}\,.
\end{equation*}

Let $\Cc_{\Lyap}(\RR^{d})$ denote the Banach space of functions in $\Cc(\RR^{d})$
with norm $\norm{f}_{\Lyap} \df \sup_{x\in\RR^{d}}
\babs{\frac{f(x)}{\Lyap(x)}}$.
Recall that a skeleton of a continuous-time Markov process is
a discrete-time Markov process with transition probability
$\Tilde{P}= \int_{0}^{\infty} \alpha(\D{t})P^{t}$, where $\alpha$
is a probability measure on $(0,\infty)$.
Since the diffusion is non-degenerate, any skeleton of the process
is $\phi$-irreducible, with an irreducibility measure absolutely
continuous with respect to the Lebesgue measure.
It is also straightforward to show that compact subsets of $\RR^{d}$ are
petite.
It then follows that for any $v\in\Ussr$ the controlled process
under $v$ is $\Lyap$-geometrically ergodic
(see \cite{Down-95,Fort-05}), or in other words
there exist constants $C_{0}$ and $\gamma>0$ such that
if $h\in\Cc_{\Lyap}(\RR^{d})$ then
\begin{equation*}
\left|P^{t}_{v}h(x) - \int_{\RR^{d}}h(x)\,\imeas_{v}(\D{x})\right|
\le C_{0} \E^{-\gamma t}
\bnorm{h}_{\Lyap} \Lyap(x)\,,\quad t\ge0\,,~x\in\RR^{d}\,.
\end{equation*}

Concerning the ergodic control problem the following result is standard
\cite{book}.

\begin{theorem}
Under Assumption~\ref{A01} there exists a unique solution
$V^{*}\in\Cc_{\Lyap}(\RR^{d})\cap\Cc^{2}(\RR^{d})$,
satisfying $V^{*}(0)=0$,  of
\begin{equation}\label{E-HJB}
0 = \min_{u\in\Act}\; \left[\Lg^{u} V^{*}(x) +r(x,u)\right]-\beta\,.
\end{equation}
A control $v^{*}\in\Usr$ is optimal with respect to the criteria
\eqref{E-ergcrit} and \eqref{E-avgcrit} if and only if it satisfies
\begin{equation}\label{e-HJBmin}
\min_{u\in\Act}\; \left[ \sum_{i=1}^{d} b^{i}(x,u) 
\frac{\partial V}{\partial x_{i}}(x) + r(x,u) \right]
= \sum_{i=1}^{d} b^{i}\bigl(x,v^{*}(x)\bigr)
\frac{\partial V}{\partial x_{i}}(x)+ r\bigl(x,v^{*}(x)\bigr)
\end{equation}
a.e.\ in $\RR^{d}$.
\end{theorem}

For the rest of the paper $v^{*}\in\Ussr$
denotes some fixed control satisfying \eqref{e-HJBmin}.

\subsection{The relative value iteration}

We study the following relative value iteration (RVI) scheme:
\begin{equation}\label{E-RVI}
\frac{\partial V}{\partial t}(t,x)
= \min_{u\in\Act}\; \left[\Lg^{u} V(t,x) +r(x,u)\right]
- V(t,0)\,,\quad V(0,x)=V_{0}(x)\,,
\end{equation}
with the boundary condition $V_{0}\in\Cc_{\Lyap}(\RR^{d})\cap\Cc^{2}(\RR^{d})$.

The main theorem of the paper is as follows.

\begin{theorem}\label{Tmain}
For each $V_{0}\in\Cc_{\Lyap}(\RR^{d})\cap\Cc^{2}(\RR^{d})$,
the solution $V(t,x)$ of \eqref{E-RVI} converges to
$V^{*}(x)+\beta$ as $t\to\infty$.
\end{theorem}

The proof of convergence of \eqref{E-RVI} is facilitated
by the study of the \emph{value iteration} (VI) equation
\begin{equation}\label{E-VI}
\frac{\partial \Bar{V}}{\partial t}(t,x)
= \min_{u\in\Act}\; \left[\Lg^{u} \Bar{V}(t,x) +r(x,u)\right]
- \beta\,,\quad \Bar{V}(0,x)=V_{0}(x)\,.
\end{equation}
Here $V_{0}\in\Cc_{\Lyap}(\RR^{d})\cap\Cc^{2}(\RR^{d})$ as in \eqref{E-RVI}.
Also $\beta$ is as in \eqref{E-beta}, so it is assumed known.

As shown in Lemma~\ref{L-bound} in Section~\ref{S-main},
$\Bar{V}(t,\cdot)$ is bounded in $\Cc_{\Lyap}(\RR^{d})$ uniformly
in $t\ge0$.
By \eqref{E-VI} we have
\begin{equation}\label{E-ident}
\Bar{V}(t,x) = \inf_{U\in\Uadm}\;
\left(\int_{0}^{t} \Exp_{x}^{U}\left[r(X_{s},U_{s})-\beta\right]\,\D{s}
+ \Exp_{x}^{U}\left[V_{0}(X_{t})\right]\right)\,.
\end{equation}
Also, as we show in Lemma~\ref{L-VV},
\begin{equation*}
V(t,x) = \Bar{V}(t,x) - \E^{-t}\int_{0}^{t}\E^{s}\Bar{V}(s,0)\,\D{s} +
\beta(1-\E^{-t})\qquad \forall x\in\RR^{d}\,,~t\ge0\,.
\end{equation*}
It follows that $V(t,\cdot)$ is also bounded in $\Cc_{\Lyap}(\RR^{d})$ uniformly
in $t\ge0$.
Additionally, convergence of $\Bar{V}(t,\cdot)$ as $t\to\infty$
implies the analogous convergence of $V(t,\cdot)$.
In Section~\ref{S-main} we provide two separate proofs of convergence
of $\Bar{V}(t,\cdot)$ as $t\to\infty$ to a solution of \eqref{E-HJB}.
The first employs results from the theory of monotone dynamical systems,
while the second utilizes a reverse martingale convergence theorem.

\begin{remark}
Note that by \eqref{E-ident} convergence of $\Bar{V}(t,\cdot)$
as $t\to\infty$ to a solution of \eqref{E-HJB}
implies that $F(t,\cdot)$ defined by
\begin{equation*}
F(t,x)\df \inf_{U\in\Uadm}\;
\int_{0}^{t} \Exp_{x}^{U}\left[r(X_{s},U_{s})-\beta\right]\,\D{s}
\end{equation*}
also converges to a solution of the HJB equation in \eqref{E-HJB}.

Note also that the (VI)
provides a sharp bound for the performance of an optimal ergodic control
$v^{*}$ over a finite horizon.
Indeed, by \eqref{E-HJB}, we have
\begin{equation*}
V^{*}(x) =
\Exp_{x}^{v^{*}}\left[\int_{0}^{t}\left(r(X_{s},U_{s})-\beta\right)\,\D{s}\right]
+ \Exp_{x}^{v^{*}}\left[V^{*}(X_{t})\right]\,.
\end{equation*}
Therefore, by \eqref{E-VI} with boundary condition $V_{0}\equiv0$, we obtain
\begin{equation}\label{E-bias}
\int_{0}^{t} \Exp^{v^{*}}_{x}\left[r(X_{t},U_{t})\right]
-\inf_{U\in\Uadm}\; \int_{0}^{t} \Exp^{U}_{x}\left[r(X_{t},U_{t})\right]
= V^{*}(x)-\Exp^{v^{*}}_{x}\left[V^{*}(X_{t})\right] - \Bar{V}(t,x)\,,
\end{equation}
and the infimum is realized by any measurable selector from the minimizer
of the (VI).
Since the right hand side of \eqref{E-bias} is bounded in
$\Cc_{\Lyap}(\RR^{d})$ uniformly in $t\in[0,\infty)$, it follows that,
under Assumption~\ref{A01}, a stationary Markov average-cost
optimal control $v^{*}$ satisfies
\begin{equation*}
\int_{0}^{T} \Exp_{x}^{v^{*}}\left[r(X_{t},U_{t})\right]\,\D{t}
\le K_{0} \Lyap(x) + \inf_{U\in\Uadm}\;
\int_{0}^{T} \Exp_{x}^{U}\left[r(X_{t},U_{t})\right]\,\D{t}\qquad
\forall T\ge0\,.
\end{equation*}
This provides a sharp bound for bias and overtaking optimality
over the class of all Markov controls
(compare with the results in 
\cite{Jasso-09} which are restricted to the class
of optimal stationary Markov controls).
\end{remark}

\section{Preliminaries}\label{S3}

\subsection{A Result from Monotone Dynamical Systems}\label{S-dyn}

Let $\cH$ be a subset of a metric space $\cY$ of real valued functions defined
on a set $\cX$.
Suppose also that $\cH$ is a subset of a Banach space $\cG$
with a positive cone $\cG_{+}$ which
has a nonempty interior.
Let $\preceq$ be the natural partial order on
$\cH$ relative to the positive cone of $\cG_{+}$.
In other words, for $h, h'\in\cH$ we write $h\preceq h'$ if $h'-h\in\cG_{+}$
for all $x\in\cX$.
We also introduce the relation
$\prec\!\!\!\prec$ and  write $h\prec\!\!\!\prec h'$ if
$h'-h\in\Int(\cG_{+})$, where `$\Int$' denotes the interior.

Let $\Phi : \cH\times\RR_{+}\to\cH$ be a semiflow on $\cH$.
In other words, $\Phi$ satisfies
\begin{romannum}
\item
$\Phi_{0}(h) = h$ for all $h\in\cH$;
\item
$\Phi_{t}\circ\Phi_{s} = \Phi_{t+s}$ for all $t$, $s\in\RR_{+}$.
\end{romannum}

As well known, if $h\in\cH$, then its
\emph{orbit} $O(h)$ is
defined by $O(h) \df \{\Phi_{t}(h) : t\ge 0\}$.
Also the
\emph{$\omega$-limit set} of $h\in\cH$ is denoted by $\omega(h)$
and defined as
$\omega(h) \df \cap_{t>0}\;\overline{\cup_{s\ge t}\; \Phi_{t}(h)}$, where
the closure is in $\cY$.
The semiflow is called \emph{monotone} (\emph{strongly monotone}) if
$h\preceq h'$ ($h\prec h'$)
implies that $\Phi_{t}(h)\preceq\Phi_{t}(h')$
($\Phi_{t}(h)\prec\!\!\!\prec\Phi_{t}(h')$) for all $t>0$.
It is called \emph{eventually strongly monotone}
if it is monotone and whenever $h\prec h'$ there exists some $t_{0}\in\RR_{+}$
such that $\Phi_{t_{0}}(h)\prec\!\!\!\prec\Phi_{t_{0}}(h')$.
As shown in \cite[Proposition~1.1]{Smith-95}, if $\Phi$ is
eventually strongly monotone then it is
\emph{strongly order preserving} (SOP), and this means that
whenever $h\prec h'$ there
exist open neighborhoods $U$ and $U'$ of $h$ and $h'$, respectively,
and $t_{0}>0$ such
that $\Phi_{t}(U) \preceq \Phi_{t}(U')$ for all $t>t_{0}$.

Let
\begin{equation*}
\mathscr{E}\df\{h\in\cH : \Phi_{t}(h)=h\,,~\forall t\ge0\}\,.
\end{equation*}
In other words, $\mathscr{E}$ is the set of equilibria of the semiflow.
A point $h\in\cH$
is called \emph{quasiconvergent} if $\omega(h)\subset\mathscr{E}$,
and \emph{convergent} if $\omega(h)$ is a singleton.
Let $\mathfrak{Q}$ and $\mathfrak{C}$ denote the sets of quasiconvergent
and convergent points, respectively.

We quote the following theorem \cite[Theorem~4.3 and Remark~4.2]{Smith-95}
which shows that quasiconvergence is generic.
We need the following notation:
We write $h_{n}\uparrow\uparrow h$ ($h_{n}\downarrow\downarrow h$)
if $h_{n}\prec h_{n+1}$ ($h_{n}\succ h_{n+1}$) and
$\lim_{n} h_{n} \to h$ in $\cH$.

\begin{theorem}\label{T3.1}
Let $\Phi_{t}$ be a strongly preserving semiflow on $\cH\subset\cG$.
Suppose that
\begin{romannum}
\item
For any $h\in\cH$ there exists a sequence $\{h_{n}\}\subset\cH$ such
that $h_{n}\uparrow\uparrow h$ or $h_{n}\downarrow\downarrow h$.
\item
For each $h\in\cH$ the closure of $O(h)$ is a compact subset of $\cH$.
\item
If $\{h_{n}\}\subset\cH$ is such that $h_{n}\uparrow\uparrow h$ or
$h_{n}\downarrow\downarrow h$, then
$\left\{\cup_{n\in\NN}\;\omega(h_{n})\right\}$ has compact closure in $\cY$
which is contained in $\cH$.
\end{romannum}
Then
$\cH = \Int(\mathfrak{Q}) \cup \overline{\Int(\mathfrak{C})}$.
Moreover, if $\mathscr{E}$ is totally ordered with respect to $\preceq$,
then $\mathfrak{Q}=\mathfrak{C}$ which implies
that $\cH = \overline{\Int(\mathfrak{C})}$.
\end{theorem}

\subsection{The case of continuous time controlled Markov chains}

To illustrate our approach, we consider here the simple
case of a controlled Markov chain
with state space $\cS$ in continuous time, with
`rate matrix' $Q^u=\left[q_{ij}(u)\right]$, $i,j\in\cS$,
depending continuously on
a parameter $u$ that lives in a compact action space $\Act$.
The matrix $Q$ satisfies $q_{ij}\ge 0$ for all $i\ne j$
and $\sum_{j\in\cS} q_{ij} = 0$.
Suppose first that the state space is finite, i.e., $\cS=\{1,\dotsc,N\}$.
To guarantee irreducibility we assume that there exists an irreducible
rate matrix $\Bar{Q}=[q_{ij}]$ and a constant $\delta>0$ such
that $q_{ij}(u) \ge \delta \Bar{q}_{ij}$ for all $i\ne j$ and $u\in\Act$.
Let $r:\cS\times\Act\to\RR$ be a running cost.
The solution of the ergodic control problem has the following characterization:
There exists a unique pair $(V^{*},\beta)$ with $\beta$ a constant and
$V^{*}:\cS\to\RR$, satisfying $V^{*}(N)=0$, which solve with $V=V^{*}$ the
equation
\begin{equation}\label{E-HJBchain}
\min_{u\in\Act}\; \left[\sum_{j\in\cS} q_{ij}(u) V(j) + r(i,u)\right]=\beta
\qquad\forall i\in\cS\,.
\end{equation}
Moreover a stationary Markov control
$v=(v_{1},\dotsc,v_{N})$ is average-cost optimal if and only
if it is a  selector from the minimizer in \eqref{E-HJBchain}.
Expressing $r$ in vector form as $r(u)=\bigl(r(1,u),\dotsc,r(N,u)\bigr)\transp$,
the relative value iteration algorithm takes the form of the
following differential equation in $\RR^{N}$:
\begin{equation}\label{E-RVIchain}
\frac{\D{h}}{\D{t}} = \min_{u\in\Act}\; \left[Q(u) h + r(u)\right]
-\bm{1}h_{N}(t)\,,
\quad h(0)=g\in\RR^{N}.
\end{equation}
where $\bm{1}$ indicates the vector whose components are all equal to $1$.
Showing existence of solutions to \eqref{E-RVIchain} is straightforward.
One can follow for example the method used in the proof
of Lemma~\ref{L-exist} which appears in Section~\ref{S-main}.
The corresponding value iteration equation is
\begin{equation}\label{E-VIchain}
\frac{\D{\Bar{h}}}{\D{t}}
= \min_{u\in\Act}\; \left[Q(u) \Bar{h} + r(u)\right]-\bm{1}\beta\,,
\quad h(0)=g\in\RR^{N}.
\end{equation}
We apply Theorem~\ref{T3.1} to \eqref{E-VIchain}.
Here $\cH$ and $\cG$ are isomorphic to $\RR^{N}$ under the Euclidean norm
topology.
Hence the partial ordering is $\Bar{h}\preceq \Bar{h}' \Longleftrightarrow
\Bar{h}_{i}\le \Bar{h}'_{i}$ for all $i\in\cS$.
The fact that \eqref{E-VIchain} is strongly order preserving follows
from the irreducibility of the chain.
Hypothesis (i) of Theorem~\ref{T3.1} is obviously satisfied in $\cH\sim\RR^{N}$.
Since the solution of \eqref{E-VIchain} is uniformly bounded for
any initial condition with the bound depending continuously on
the initial condition $g$, it follows that hypotheses (ii) and (iii) of
Theorem~\ref{T3.1} are satisfied.
The equilibrium set $\mathscr{E}$
of \eqref{E-VIchain} is the set of $V\in\RR^{N}$ which
solve \eqref{E-HJBchain}.
Hence $\mathscr{E}=\{V^{*}+c: c\in\RR\}$, which is a totally ordered set.
It then follows from Theorem~\ref{T3.1} that
$\cH = \overline{\Int(\mathfrak{C})}$.
It is also straightforward
to show from \eqref{E-VIchain} that the solutions are continuous
with respect to the initial condition, uniformly in $t\in[0,\infty)$,
i.e., that if $g^{n}$ is a sequence converting $g\in\RR^{N}$ as $n\to\infty$,
then
\begin{equation*}
\sup_{t\ge0}\; \babs{\Phi_{t}(g^{n})-\Phi_{t}(g)}\xrightarrow[n\to\infty]{}0\,.
\end{equation*}
As a result, $\mathfrak{C}$ is closed and hence every initial condition
is convergent point.
By \eqref{E-RVIchain}--\eqref{E-VIchain} and following the argument
at the end of Section~\ref{S4.1} for the
proof of Theorem~\ref{Tmain}, it follows
that $h(t)$ converges to $V^{*}+\beta$.
Convergence of the relative value iteration for countable state space
Markov chains in continuous time follows along the same lines, provided
a Lyapunov hypothesis analogous to \eqref{e-lyap} is imposed,
as well as appropriate assumptions
to guarantee the regularity of the process.
We don't delve into these details, since the focus in this paper
is continuous state space models.

\subsection{Notation and Background}\label{S-not}

The term \emph{domain} in $\RR^{d}$
refers to a nonempty, connected open subset of the Euclidean space $\RR^{d}$.
We introduce the following notation for spaces of real-valued functions on
a domain $D\subset\RR^{d}$.
The space $\Lp^{p}(D)$, $p\in[1,\infty)$, stands for the usual Banach space
of (equivalence classes) of measurable functions $f$ satisfying
$\int_{D} \abs{f(x)}^{p}\,\D{x}<\infty$, and $\Lp^{\infty}(D)$ is the
Banach space of functions that are essentially bounded in $D$. 
The space $\Cc^{k}(D)$ ($\Cc^{\infty}(D)$)
refers to the class of all functions whose partial
derivatives up to order $k$ (of any order) exist and are continuous.
The standard Sobolev space of functions on $D$ whose generalized
derivatives up to order $k$ are in $\Lp^{p}(D)$, equipped with its natural
norm, is denoted by $\Sob^{k,p}(D)$, $k\ge0$, $p\ge1$.

We adopt the notation
$\partial_{t}\df\tfrac{\partial}{\partial{t}}$, and for $i,j\in\NN$,
$\partial_{i}\df\tfrac{\partial~}{\partial{x}_{i}}$ and
$\partial_{ij}\df\tfrac{\partial^{2}~}{\partial{x}_{i}\partial{x}_{j}}$.
We often use the standard summation rule that
repeated subscripts and superscripts are summed from $1$ through $d$.

\subsection{Some Facts from Parabolic Equations}

For a nonnegative multi-index $\alpha=(\alpha_{1},\dotsc,\alpha_{d})$
we let $D^{\alpha}\df \partial_{1}^{\alpha_{1}}\dotsb
\partial_{d}^{\alpha_{d}}$.
Let $Q$ be a domain in $\RR_{+}\times\RR^{d}$.
Recall that $\Cc^{r,k+2r}(Q)$ stands for the set of bounded
continuous functions $\varphi(t,x)$ defined on $Q$ such
that the derivatives $D^{\alpha}\partial_{t}^{\ell}\varphi$
are bounded and continuous in $Q$ for
\begin{equation}\label{e-indices}
\abs{\alpha}+2\ell \le k+2r\,,\qquad \ell\le r\,.
\end{equation}
For $\varphi\in\Cc^{r,k+2r}(\Bar{Q})$ and $p\in[1,\infty)$, define
\begin{equation*}
\bnorm{\varphi}_{\Sob^{r,k+2r,p}(Q)} \df
\sum_{\substack{\abs{\alpha}\le k+2(r-\ell)\\
\ell\le r}} \bnorm{D^{\alpha}\partial_{t}^{\ell}\varphi}_{\Lp^{p}(Q)}\,.
\end{equation*}
The \emph{parabolic Sobolev space} $\Sob^{r,k+2r,p}(Q)$
is the subspace of $\Lp^{p}(Q)$ which consists of those functions
$\varphi$ for which there exists a sequence $\varphi_{n}$ in
$\Cc^{r,k+2r}(\Bar{Q})$ such that $\bnorm{\varphi_{n}-\varphi}_{\Lp^{p}(Q)}\to0$
as $n\to\infty$ and
\begin{equation*}
\bnorm{D^{\alpha}\partial_{t}^{\ell}\varphi_{n}
-D^{\alpha}\partial_{t}^{\ell}\varphi_{m}}_{\Lp^{p}(Q)}
\xrightarrow[n,m\to\infty]{}0\,,
\end{equation*}
for all $\alpha$ and $\ell$ satisfying \eqref{e-indices}.
In this way the Sobolev derivatives
$D^{\alpha}\partial_{t}^{\ell}\varphi$ are well defined
as functions in $\Lp^{p}(Q)$ and 
$\Sob^{r,k+2r,p}(Q)$ is a Banach space under the norm introduced.

Let $r:\RR^{d}\times\Act$ be a nonnegative continuous function which is
locally Lipschitz continuous in $x$ uniformly in $u\in\Act$.
Let $\kappa_{R}$ be a Lipschitz constant of $r$ over $B_{R}$.

We next review some standard estimates for solutions of equations
of the form
\begin{equation}\label{e-quasi}
-\partial_{t}\varphi(t,x) +
\min_{u\in\Act}\; \left[\Lg^{u} \varphi(t,x) +
r(x,u)\right] = f(t,x)
\end{equation}
and
\begin{equation}\label{e-lin1}
-\partial_{t}\varphi(t,x) + \Lg^{v_{t}} \varphi(t,x) = g(t,x)\,.
\end{equation}
Note that if $v$ is a measurable selector from the minimizer
in \eqref{e-quasi} then the quasilinear equation
\eqref{e-quasi} transforms to the linear equation
\eqref{e-lin1},
which in fact takes the particular form
\begin{equation}\label{e-lin2}
-\partial_{t}\varphi(t,x) + a^{ij}\partial_{ij}\varphi(t,x) +
H(D\varphi,x) = f(t,x)\,,
\end{equation}
where $H$ is Lipschitz continuous in its arguments.

For $R>0$ and $0\le T'<T$ define $B_{R}^{T',T}\df (T',T)\times B_{R}$.
Let $g\in\Sob^{0,k,p} \bigl(B_{R}^{0,T}\bigr)$ and
suppose that $\varphi\in\Sob^{1,2,p}\bigl(B_{R}^{0,T}\bigr)$
is a solution of \eqref{e-lin1}.
Then for any $R'\in(0,R)$ and $T'\in(0,T)$ it holds that
$\varphi\in\Sob^{1,2+k,p}\bigl(B_{R'}^{T',T}\bigr)$  and there exists
a constant $C_{1}=C_{1}(R',R,T',T,k,d,\kappa_{R},\kappa_{1},p)$ such that
\begin{equation}\label{est01}
\bnorm{\varphi}_{\Sob^{1,2+k,p}\bigl(B_{R'}^{T',T}\bigr)}
\le C_{1}\left(\bnorm{g}_{\Sob^{0,k,p}\bigl(B_{R}^{0,T}\bigr)}
+\bnorm{\varphi}_{\Lp^{p}\bigl(B_{R}^{0,T}\bigr)}\right)\,.
\end{equation}
Combining \eqref{est01} with the compactness of the imbedding
of $\Sob^{2,p}(B_{R})\hookrightarrow \Cc^{1}(\Bar{B}_{R})$, for $p>d$,
and the interpolation inequality,
we conclude by using \eqref{e-lin2} that if
$f\in\Sob^{0,1,p}\bigl(B_{R}^{0,T}\bigr)$,
then $\varphi\in\Cc^{1,2}\bigl(B_{R'}^{T',T}\bigr)$ and
\begin{equation}\label{est02}
\max_{ \abs{\alpha}\le 2}\;\sup_{B_{R'}^{T',T}}\;
\abs{D^{\alpha}\varphi}
\le C_{2}\left(\bnorm{f}_{\Sob^{0,k,p}\bigl(B_{R}^{0,T}\bigr)}
+\bnorm{\varphi}_{\Lp^{p}\bigl(B_{R}^{0,T}\bigr)}\right)\,,
\end{equation}
where $C_{2}$ depends on the parameters in $C_{1}$.
Moreover, if the derivatives $\partial_{i}f$ are bounded on
$B_{R}^{0,T}$ then
\begin{equation}\label{est03}
\max_{ \abs{\alpha}\le 1}\;\sup_{B_{R'}^{T',T}}\;
\abs{D^{\alpha}\partial_{t}\varphi}
\le C_{3}\left(
\max_{ \abs{\alpha}\le 1}\;\sup_{B_{R'}^{T',T}}\;
\abs{D^{\alpha}f} +
\bnorm{f}_{\Sob^{0,k,p}\bigl(B_{R}^{0,T}\bigr)}
+\bnorm{\varphi}_{\Lp^{p}\bigl(B_{R}^{0,T}\bigr)}\right)\,.
\end{equation}
These estimates can be found in \cite[Chapter~5]{Krylov-08}.

\section{Main Results}\label{S-main}

\subsection{Proof of Theorem~\ref{Tmain}}\label{S4.1}

The proof of Theorem~\ref{Tmain} involves several intermediate results.
For a subset $Q$ of $\RR_{+}\times\RR^{d}$, by abuse of notation,
we let $\Cc_{\Lyap}(Q)$ denote the Banach space of functions
in $\Cc(Q)$ with norm
\begin{equation*}
\norm{f}_{\Lyap} \df \sup_{(t,x)\in Q}\;
\babs{\tfrac{f(t,x)}{\Lyap(x)}}\,.
\end{equation*}
Let $\RR^{d}_{T}\df[0,T]\times\RR^{d}$.
We next show that \eqref{E-RVI} has a unique solution
in $\Cc_{\Lyap}(\RR^{d}_{T})\cap\Cc^{1,2}(\RR^{d}_{T})$, for any $T>0$.

\begin{lemma}\label{L-exist}
For each $V_{0}\in\Cc_{\Lyap}(\RR^{d})\cap\Cc^{2}(\RR^{d})$,
there exists a unique solution
$V\in\Cc_{\Lyap}(\RR^{d}_{T})\cap\Cc^{1,2}(\RR^{d}_{T})$, for
any $T>0$.
\end{lemma}

\begin{proof}
We first show that if
$g:[0,T]\to \RR^{d}$ is a bounded continuous function, then
\begin{equation}\label{e-aux0}
\partial_{t}\varphi(t,x)
= \min_{u\in\Act}\; \left[\Lg^{u} \varphi(t,x) +r(x,u)\right]
- g(t)\,,\quad \varphi(0,x)=V_{0}(x)
\end{equation}
has a unique solution in
$\Cc_{\Lyap}(\RR^{d}_{T})\cap\Cc^{1,2}(\RR^{d}_{T})$.

Let $r^{n}$ denote the truncation of $r$, i.e.,
$r^{n}(c,u) \df n\wedge r(x,u)$.
Let $\uptau_{R}$ denote the first exit time from the ball of radius $R$
centered at the origin in $\RR^{d}$, and let $\psi_{R}:\RR^{d}\to[0,1]$ be
a smooth function which satisfies $\psi_{R}(x)=1$ for
$\abs{x}\le\nicefrac{R}{2}$ and $\psi_{R}(x)=0$ for
$\abs{x}\ge\nicefrac{3R}{4}$.
Then the boundary value problem
\begin{equation}\label{e-aux1}
\begin{gathered}
\partial_{t} \varphi_{n,R}(t,x)
= \min_{u\in\Act}\; \left[\Lg^{u} \varphi_{n,R}(t,x) +r^{n}(x,u)\right]
- g(t)\,,\quad \\[5pt]
\varphi_{n,R}(0,x)=V_{0}(x)\psi_{R}(x)\,,\qquad
\varphi_{n,R}(t,\cdot\,)\vert_{\partial B_{R}}= 0\quad\forall t\ge0\,,
\end{gathered}
\end{equation}
has a unique solution in $\Cc^{1,2}(\RR^{d}_{T})$.
This solution has the stochastic representation
\begin{equation}\label{e-repr1}
\varphi_{n,R}(t,x)
= \inf_{U\in\Uadm}\;
\Exp^{U}_{x}\left[V_{0}(X_{t})\psi_{R}(X_{t})\Ind\{t<\uptau_{R}\}
+\int_{0}^{t\wedge\uptau_{R}}
\left(r^{n}(X_{s},U_{s})-g(s)\right)\,\D{s}\right]\,,
\end{equation}
where $\Ind$ denotes the indicator function.
Since
\begin{align*}
\int_{0}^{t} \Exp^{U}_{x}\left[r^{n}(X_{s},U_{s})\right]\,\D{s}
&\le c_{2} \int_{0}^{t} \Exp^{U}_{x}\left[\Lyap(X_{s})\right]\,\D{s}\\
&\le \frac{c_{2}}{c_{1}}\left( c_{0} t + \Lyap(x)\right)
\qquad \forall U\in\Uadm\,,
\end{align*}
and $V_{0}\in\Cc_{\Lyap}(\RR^{d})$,
we obtain
\begin{equation}\label{e-bound1}
\varphi_{n,R}(t,x)\le c_{3} (1+\Lyap(x))
+ \frac{c_{2}}{c_{1}}\left( c_{0} t + \Lyap(x)\right)
+ \norm{g}_{\Lp^{1}([0,T])}
\end{equation}
for some constant $c_{3}>0$.
Also by \eqref{e-repr1} we have
\begin{equation}\label{e-bound2}
\varphi_{n,R}(t,x)\ge - c_{3} (1+\Lyap(x))-\int_{0}^{t} g(s)\,\D{s}\,,
\end{equation}
and it follows
that for any fixed $g$ and $V_{0}$, the solution
$\varphi_{n,R}$ is bounded in $\Cc_{\Lyap}(\RR^{d}_{T})$ uniformly in
$R>0$ and $n\in\NN$.
The interior estimates of solutions of \eqref{e-aux1} (see \cite[p.~342
and p.~351]{Lady})
allow us to take limits as $R\to\infty$ (along some subsequence)
to obtain a solution $\varphi_{n}\in\Cc^{1,2}(\RR^{d}_{T})$ to
\begin{equation}\label{e-aux2}
\partial_{t}\varphi_{n}(t,x)
= \min_{u\in\Act}\; \left[\Lg^{u} \varphi_{n}(t,x) +r^{n}(x,u)\right]
- g(t)\,,\quad \varphi_{n}(0,x)=V_{0}(x)\,,
\end{equation}
which naturally satisfies the bounds in \eqref{e-bound1}--\eqref{e-bound2}.
Using again the interior estimates of solutions
to \eqref{e-aux2} we can let $n\to\infty$
to obtain in the limit a solution
$\varphi\in\Cc_{\Lyap}(\RR^{d}_{T})\cap\Cc^{1,2}(\RR^{d}_{T})$ to \eqref{e-aux0}.
Showing uniqueness of this solution is standard.
Let $\varphi$ and $\varphi'$ be such solutions of \eqref{e-aux0}
corresponding to $g$ and $g'$, respectively.
Using the inequality $\abs{\inf\;A-\inf\;B}\le \sup\;\abs{A-B}$
we have
\begin{align*}
\sup_{t\in[0,T]}\;\babs{\varphi(t,0)-\varphi'(t,0)} &\le \sup_{t\in[0,T]}\;\left|
\inf_{U\in\Uadm}\;\Exp_{0}^{U}\left[
V_{0}(X_{t}) + \int_{0}^{t} \bigl(r(X_{s},U_{s})-g(s)\bigr)\,\D{s}\right]\right.\\
&\mspace{150mu}\left.
- \inf_{U\in\Uadm}\;\Exp_{0}^{U}\left[
V_{0}(X_{t}) + \int_{0}^{t} \bigl(r(X_{s},U_{s})-g'(s)\bigr)\,\D{s}\right]
\right|\\[5pt]
&\le \sup_{t\in[0,T]}\;\left|
\int_{0}^{t} \bigl[g(s)-g'(s)\bigr]\,\D{s}\right|\\[5pt]
&\le T\sup_{t\in[0,T]}\;\babs{g(t)-g'(t)}\,.
\end{align*}
Hence for $T<1$ the map $g(\cdot)\mapsto \varphi(\cdot,0)$
is a contraction thus asserting the
existence of a solution to \eqref{E-RVI} in
$\Cc_{\Lyap}(\RR^{d}_{T})\cap\Cc^{1,2}(\RR^{d}_{T})$, for $T<1$.
Concatenating intervals $[0,T]$, $[T,2T]$, $\dotsc$, with $T<1$,
we obtain such a solution
of \eqref{E-RVI} for any $T>0$.
Uniqueness is again standard.
\end{proof}

The next two lemmas concern estimates for the solutions of the
(RVI) and the (VI).

\begin{lemma}\label{L-bound}
For each $V_{0}\in\Cc_{\Lyap}(\RR^{d})\cap\Cc^{2}(\RR^{d})$,
the solution $\Bar{V}$ of \eqref{E-VI} satisfies
the bound
\begin{equation}\label{EL-bound}
\babs{V^{*}(x) - \Bar{V}(t,x)} \le \bnorm{V^{*}-V_{0}}_{\Lyap}
\left(\frac{c_{0}}{c_{1}} + \Lyap(x) \E^{-c_{1} t}\right)
\qquad \forall x\in\RR^{d}\,,~\forall t\ge0\,.
\end{equation}
\end{lemma}

\begin{proof}
Let $v^{*}$ be a measurable selector from the minimizer in \eqref{E-HJB}.
Then
\begin{equation}\label{e-t01.0}
-\partial_{t}(V^{*}-\Bar{V}) + \Lg^{v^{*}}(V^{*}-\Bar{V})\le 0
\end{equation}
from which, by an application of It\^o's formula to
$V^{*}(X_{s}) - \Bar{V}(t-s, X_{s})$, $s\in[0,t]$, it follows that
\begin{equation}\label{e-t01.1}
\Exp_{x}^{v^{*}}\left[V^{*}(X_{t}) -V_{0}(X_{t})\right]
\le V^{*}(x) - \Bar{V}(t,x)\,.
\end{equation}
On the other hand, if $\Bar{v}$ is a measurable selector from the
minimizer in \eqref{E-VI}, then
\begin{equation*}
-\partial_{t}(V^{*}-\Bar{V}) + \Lg^{\Bar{v}}(V^{*}-\Bar{V})\ge 0\,,
\end{equation*}
and we obtain
\begin{equation}\label{e-t01.2}
V^{*}(x) - \Bar{V}(t,x)\le
\Exp_{x}^{\Bar{v}}\left[V^{*}(X_{t}) -V_{0}(X_{t})\right]\,.
\end{equation}
Since $V^{*}$ and $V_{0}$ are in $\Cc_{\Lyap}(\RR^{d})$, 
\eqref{EL-bound} follows by
\eqref{e-estlyap} and \eqref{e-t01.1}--\eqref{e-t01.2}.
\end{proof}

\begin{remark}\label{R4.3}
Note that the Markov control associated  with
a measurable selector $\Bar{v}$ from the
minimizer in \eqref{E-VI} is computed `backward' in time.
Hence the control applied to the process $X$ considered
in \eqref{e-t01.2}
is the Markov control $U(s,x)=\Bar{v}(t-s,x)$, $0\le s\le t$,
where $\Bar{v}$ solves
\begin{equation*}
\partial_{t}\Bar{V}(t,x) = a^{ij}(x)\partial_{ij}\Bar{V}(t,x) +
b^{i}\bigl(x,\Bar{v}(t,x)\bigr)\partial_{i}\Bar{V}(t,x)
+ r\bigl(x,\Bar{v}(t,x)\bigr) -\beta\,.
\end{equation*}
\end{remark}

\begin{lemma}\label{L-VV}
If $\Bar{V}(0,x)=V(0,x)=V_{0}(x)$
for some $V_{0}\in\Cc_{\Lyap}(\RR^{d})\cap\Cc^{2}(\RR^{d})$,
then the solutions $V$ and $\Bar{V}$ of \eqref{E-RVI} and \eqref{E-VI},
respectively,
satisfy
\begin{align}
V(t,x) - V(t,0) &= \Bar{V}(t,x) - \Bar{V}(t,0)\label{EL-VV01}\\[5pt]
V(t,x) &= \Bar{V}(t,x) - \E^{-t}\int_{0}^{t}\E^{s}\Bar{V}(s,0)\,\D{s} +
\beta(1-\E^{-t})\label{EL-VV02}
\end{align}
for all $x\in\RR^{d}$ and all $t\ge0$.
\end{lemma}

\begin{proof}
By \eqref{E-RVI} and \eqref{E-VI} we have
\begin{subequations}
\begin{align}
V(t,x) &= \inf_{U\in\Uadm}\;
\left(\int_{0}^{t} \Exp_{x}^{U}\left[r(X_{s},U_{s})\right]\,\D{s}
+ \Exp_{x}^{U}\left[V_{0}(X_{t})\right]\right)
-\int_{0}^{t}V(s,0)\,\D{s}\label{EL-VV03a}\\[5pt]
\Bar{V}(t,x) &= \inf_{U\in\Uadm}\;
\left(\int_{0}^{t} \Exp_{x}^{U}\left[r(X_{s},U_{s})\right]\,\D{s}
+ \Exp_{x}^{U}\left[V_{0}(X_{t})\right]\right)-\beta t\label{EL-VV03b}\,.
\end{align}
\end{subequations}
Hence \eqref{EL-VV01} follows by \eqref{EL-VV03a}--\eqref{EL-VV03b}.
Again by \eqref{EL-VV03a}--\eqref{EL-VV03b} we have
\begin{equation}\label{E-VV04}
V(t,0)-\beta + \int_{0}^{t}\bigl(V(s,0)-\beta\bigr)\,\D{s} = \Bar{V}(t,0)-\beta\,,
\end{equation}
and solving \eqref{E-VV04} we obtain
\begin{equation*}
V(t,0) = \Bar{V}(t,0) - \E^{-t}\int_{0}^{t}\E^{s}\Bar{V}(s,0)\,\D{s} +
\beta(1-\E^{-t})\,,
\end{equation*}
which combined with \eqref{EL-VV01} yields \eqref{EL-VV02}.
\end{proof}

Next we show that the solution $\Bar{V}$ of the (VI) converges
as $t\to\infty$ for any initial condition $V_{0}$.

\begin{theorem}\label{t01}
For each $V_{0}\in\Cc_{\Lyap}(\RR^{d})\cap\Cc^{2}(\RR^{d})$,
$\Bar{V}(t,x)\to V^{*}(x)+c$ as $t\to\infty$, for some $c\in\RR$
which depends on $V_{0}$.
\end{theorem}

\begin{proof}
We view the solutions of \eqref{E-VI} as a semiflow
on $\cH=\cY=\Cc_{\Lyap}(\RR^{d})\cap\Cc^{2}(\RR^{d})$,
also letting $\cG=\Cc_{\Lyap}(\RR^{d})$,
and apply Theorem~\ref{T3.1}.
We equip $\Cc_{\Lyap}(\RR^{d})\cap\Cc^{2}(\RR^{d})$
with a complete metric, for example
by letting
\begin{equation*}
d(f,g) \df \bnorm{f-g}_{\Cc_{\Lyap}(\RR^{d})} + \sum_{n=1}^{\infty}\frac{1}{2^{n}}
\max\left(1,\bnorm{f-g}_{\Cc^{2}(B_{n})}\right)\,,
\end{equation*}
where $B_{n}$ denotes the ball of radius $n$ centered at the origin
in $\RR^{d}$ and
\begin{equation*}
\bnorm{f}_{\Cc^{2}(B)}\df
\sum_{\abs{\alpha}\le 2}
\sup_{B}\;\abs{D^{\alpha}f}\,.
\end{equation*}
Hypothesis (i) of Theorem~\ref{T3.1} is clearly satisfied.
Let $\Phi_{t}(V_{0}):\RR^{d}\to\RR$ denote the solution of
\eqref{E-VI} corresponding
to $V_{0}\in\Cc_{\Lyap}(\RR^{d})\cap\Cc^{2}(\RR^{d})$.
Let
$\mathscr{E}\df\{V^{*} + c : c\in\RR\}
\subset\Cc_{\Lyap}(\RR^{d})\cap\Cc^{2}(\RR^{d})$,
i.e., the set of equilibria of this semiflow.
Note the following:
\begin{enumerate}
\item[(a)]
for each $V_{0}\in\Cc_{\Lyap}(\RR^{d})\cap\Cc^{2}(\RR^{d})$,
$\Phi_{t}(V_{0})$ is bounded 
in $\Cc_{\Lyap}(\RR^{d})$ by \eqref{EL-bound}.
Also the second order partial derivatives of $\Phi_{t}(V_{0})$ are locally
equicontinuous in $x$, uniformly in $t\ge T$ for some $T>0$
(this requires a slight
improvement of \eqref{est02}, adding H\"older continuity which
is standard \cite[Theorem~5.1]{Lady}).
Hence, every subsequence $\Phi_{t_{n}}(V_{0})$
contains a further subsequence that converges
in $\Cc_{\Lyap}(\RR^{d})\cap\Cc^{2}(\RR^{d})$, 
which, in turn, implies that
the orbit $\{\Phi_{t}(V_{0}): t\in\RR_{+}\}$ has a compact closure
in $\Cc_{\Lyap}(\RR^{d})\cap\Cc^{2}(\RR^{d})$.
\item[(b)]
If $\{V_{0}^{n}\}\subset\Cc_{\Lyap}(\RR^{d})\cap\Cc^{2}(\RR^{d})$
is a monotone sequence
such that $V_{0}^{n}\to V_{0}\in\Cc_{\Lyap}(\RR^{d})\cap\Cc^{2}(\RR^{d})$
as $n\to\infty$,
then by \eqref{EL-bound} the set
$\left\{\cup_{n\in\NN}\;\Phi_{t}(V_{0}^{n}): t\in\RR_{+}\right\}$ is bounded
in $\Cc_{\Lyap}(\RR^{d})$.
Hence it has locally H\"older equicontinuous
second order partial derivatives in $x$, which implies that it
has a compact closure in $\Cc_{\Lyap}(\RR^{d})\cap\Cc^{2}(\RR^{d})$.
In particular, the set
$\left\{\cup_{n\in\NN}\;\omega(V_{0}^{n})\right\}$
has compact closure in $\Cc_{\Lyap}(\RR^{d})\cap\Cc^{2}(\RR^{d})$.
\end{enumerate}
Hence assumptions (ii) and (iii) of Theorem~\ref{T3.1} are satisfied.

Consider the partial order relation $\preceq$ on
$\Cc_{\Lyap}(\RR^{d})\cap\Cc^{2}(\RR^{d})$ induced by the positive cone
$\Cc_{\Lyap}(\RR^{d})_{+}$ in $\Cc_{\Lyap}(\RR^{d})$.
Indeed, if $V_{0}\prec V_{0}'$ then \eqref{E-VI} yields
\begin{equation}\label{e-smonotone}
\Exp_{x}^{v'}\left[V_{0}'(X_{t})-V_{0}(X_{t})\right]
\le \Phi_{t}(V_{0}')(x)-\Phi_{t}(V_{0})(x)
\qquad \forall t>0\,,\quad\forall x\in\RR^{d}\,,
\end{equation}
where $v'$ is a Markov control associated with
a measurable selector from the minimizer in \eqref{E-VI}
corresponding to the solution starting at $V_{0}'$ (see Remark~\ref{R4.3}).
It follows from \eqref{e-smonotone} and the fact that the support
of the transition probabilities of the controlled process is the entire
space $\RR^{d}$ that if $V_{0}\prec V_{0}'$, then
$\Phi_{t}(V_{0})\prec\!\!\!\prec\Phi_{t}(V_{0}')$ for all $t>0$,
or in other words that the semiflow $\Phi$ is strongly monotone on
$\Cc_{\Lyap}(\RR^{d})\cap\Cc^{2}(\RR^{d})$.
As mentioned in Section~\ref{S-dyn} the semiflow is then strongly
order-preserving.
Since $\mathcal{E}$ is totally ordered
it follows by Theorem~\ref{T3.1} that
$\cH = \overline{\Int(\mathfrak{C})}$.

It remains to show that $\mathfrak{C}$ is closed.
Note that
\begin{align*}
\babs{\Phi_{t}(V_{0})(x) - \Phi_{t}(V_{0}')(x)}
&\le\sup_{U\in\Uadm}\;\babs{\Exp_{x}^{U}\left[V_{0}(X_{t})-V_{0}'(X_{t})\right]}
\\[5pt]
&\le \bnorm{V_{0}-V_{0}'}_{\Lyap}\;
\left(\sup_{U\in\Uadm}\;\Exp_{x}^{U}\left[\Lyap(X_{t})\right]\right)\,,
\quad t\ge0\,,
\end{align*}
Hence by \eqref{e-estlyap} we have
\begin{align*}
\bnorm{\Phi_{t}(V_{0}) - \Phi_{t}(V_{0}')}_{\Lyap} &\le
\left(\sup_{x\in\RR^{d}}\;\sup_{U\in\Uadm}\;
\frac{\Exp_{x}^{U}\left[\Lyap(X_{t})\right]}{\Lyap(x)} \right)
\bnorm{V_{0} - V_{0}'}_{\Lyap}\\[5pt]
&\le \left(1+\frac{c_{0}}{c_{1}}\right)
\bnorm{V_{0} - V_{0}'}_{\Lyap}\,,\quad t\ge0\,.
\end{align*}
This shows in particular that if $V_{0}^{n}$ is a Cauchy sequence of
convergent points in $\Cc_{\Lyap}(\RR^{d})\cap\Cc^{2}(\RR^{d})$,
then $f_{n}\df\omega(V_{0}^{n})$ converges
in $\Cc_{\Lyap}(\RR^{d})$ as $n\to\infty$.
Suppose that $V_{0}\in\mathfrak{C}^{c}$.
Since $\mathfrak{C}$ is dense, there exists $\{V_{0}^{n}\}\subset\mathfrak{C}$
such that $V_{0}^{n}\to V_{0}$ as $n\to\infty$.
Let $f\df\lim_{n\to\infty}\;\omega(V_{0}^{n})$.
Since $V_{0}\in\mathfrak{C}^{c}$, then
$\limsup_{t\to\infty}\;d(\Phi_{t}(V_{0}),f)>0$.
Moreover, since for some $T>0$ the set
$\{\Phi_{t}(V_{0}) : t>T\}$ is precompact
in $\Cc_{\Lyap}(\RR^{d})\cap\Cc^{2}(\RR^{d})$
there exists $f'\in\Cc_{\Lyap}(\RR^{d})\cap\Cc^{2}(\RR^{d})$,
$f\ne f'$ and a sequence $t'_{n}$ such that
$\Phi_{t'_{n}}(V_{0})\to f'$ as $n\to\infty$.
On the other hand, we can find a sequence $t_{n}$ such that
$\sup_{t>t_{n}}\;\bnorm{\Phi_{t}(V_{0}^{n}) - f}_{\Lyap}\to0$
in $\Cc_{\Lyap}(\RR^{d})\cap\Cc^{2}(\RR^{d})$ as $n\to\infty$.
Therefore, for some subsequence $n(k)\uparrow\infty$, we have
$\bnorm{\Phi_{t_{n(k)}}(V_{0}^{k}) - f}_{\Lyap}\to0$ as $k\to\infty$.
Therefore,
\begin{align*}
0 &< \bnorm{f' - f}_{\Lyap}\\[5pt]
&=\lim_{k\to\infty}\;
\bnorm{\Phi_{t_{n(k)}}(V_{0}) - \Phi_{t_{n(k)}}(V_{0}^{k})}_{\Lyap}\\[5pt]
&\le\left(1+\frac{c_{0}}{c_{1}}\right) \lim_{k\to\infty}\;
\bnorm{V_{0} - V_{0}^{k}}_{\Lyap}\\[5pt]
&= 0\,,
\end{align*}
yielding a contradiction.
Thus we have shown that all points of
$\Cc_{\Lyap}(\RR^{d})\cap\Cc^{2}(\RR^{d})$ are convergent and
the proof is complete.
\end{proof}

We are now ready for the proof of the main result.

{\em Proof of Theorem~\ref{Tmain}}.
If we define $g(t)\df V(t,x) - \Bar{V}(t,x)$, then by
\eqref{EL-VV02} we have
\begin{equation*}
g(t) = \int_{0}^{t} \E^{s-t}\left(\beta - \Bar{V}(s,0)\right)\,\D{s}
\end{equation*}
Since $\Bar{V}(t,x)\to V^{*}(x)+c$ as $t\to\infty$ for each $x\in\RR^{d}$
by Theorem~\ref{t01}, it follows that
$g(t)$ converges to $\beta-V^{*}(0)-c=\beta-c$ as $t\to\infty$.
Hence
\begin{equation*}
\lim_{t\to\infty}\;V(t,x)=\lim_{t\to\infty}\;[g(t)+\Bar{V}(t,x)]
= V^{*}(x)+\beta\qquad \forall x\in\RR^{d}\,,
\end{equation*}
and the proof is complete.
\endproof

\subsection{An alternate proof of Theorem~\ref{t01}}
Recall that $v^{*}$ is an optimal stationary Markov control.
Let $\imeas_{v^{*}}$ be the corresponding invariant probability distribution,
and let $X^{*}_{t}$, $t\in\RR$,
be a stationary solution  of \eqref{e-sde} under the
control $v^{*}$ such that the law of $X^{*}_{t}$ is $\imeas_{v^{*}}$
for all $t\in\RR$.
Let $\sF_{t}\df \sigma( X_{s} : -\infty<s<t)$ and
\begin{equation*}
\Psi(t,x)\df\Bar{V}(t,x)- V^{*}(x)\,.
\end{equation*}
By \eqref{e-t01.0} we have
\begin{equation*}
-\partial_{t} \Psi(t,x) + \Lg^{v^{*}}\Psi(t,x)\ge 0\,.
\end{equation*}
Therefore the process
\begin{equation*}
M_{t} \df \Psi(t,X^{*}_{-t})\,,\quad t\in[0,\infty)\,,
\end{equation*}
is a reverse $\bigl(\sF_{-t}\bigr)$-supermartingale.
Also, by \eqref{EL-bound} there exists a constant $C_{0}$
such that $\Exp\left[\abs{M_{t}}\right]\le C_{0}$, for all $t\in[0,\infty)$.
We argue by contradiction.
Suppose that $\Bar{V}(t,\,\cdot\,)-V^{*}(\,\cdot\,)$ does not converge
to a constant as $t\to\infty$.
Then, there must exist
constants $a<b$,
a ball $D\subset\RR^{d}$ and a pair of sequences $\{t_{n}\}\subset\RR_{+}$
and $\{x_{n}\}\subset{D}$, $n\in\NN$, such that
\begin{equation}\label{e-mart00}
\Psi(t_{2k-1},x_{2k-1}) \le a\,,\quad
\Psi(t_{2k},x_{2k}) \ge b\qquad \forall k\in\NN\,.
\end{equation}
Let $B_{r}(x)$ denote the open ball of radius $r$ centered at $x\in\RR^{d}$.
Since $\Bar{V}(\,\cdot\,,t)-V^{*}(\,\cdot\,)$ is uniformly equicontinuous on
any bounded domain, there exists $r>0$, such that if $x\in D$, then
\begin{equation}\label{e-mart01}
\abs{\Psi(t,x)-\Psi(t,y)}\le \frac{b-a}{4}\qquad\forall y\in B_{2r}(x)\,.
\end{equation}
Let $\mathscr{G}=\{G_{i} : 1\le i\le N\}$ be a finite open cover of $D$
with balls of radius $r$.
Since $\mathscr{G}$ is finite,
an infinite number
of terms of the sequences $\{x_{2k-1} : k\in\NN\}$
and $\{x_{2k} : k\in\NN\}$ lie in some elements $G'$ and $G''$ of $\mathscr{G}$,
respectively.
Dropping to a subsequence of $\{x_{k}\}$, which is also denoted as $\{x_{k}\}$,
it follows by \eqref{e-mart00}--\eqref{e-mart01} that
\begin{equation}\label{e-mart03}
\begin{aligned}
\Psi(t_{2k-1},x) &\le a'\df \frac{3a+b}{4}\qquad \forall x\in G'\\[5pt]
\Psi(t_{2k},x) &\ge b'\df \frac{a+3b}{4}\qquad \forall x\in G''
\end{aligned}
\end{equation}
for all $k=1,2,\dotsc$.
Without loss of generality we can also assume that the time sequence $t_{n}$
satisfies $t_{n+1}-t_{n}>\gamma_{0}>0$.
The convergence of the transition probability under the control $v^{*}$ implies
that
for some constant $\varepsilon_{0}>0$
\begin{equation}\label{e-mart04}
\Prob_{x}^{v^{*}}(X^{*}_{t}\in B_{r}(y))\ge\varepsilon_{0}
\qquad\forall x,y\in D\,,\quad \forall t>\gamma_{0}\,.
\end{equation}
It follows by \eqref{e-mart03}--\eqref{e-mart04} that
\begin{equation*}
\Prob^{v^{*}}_{\imeas_{v^{*}}}\left( \sum_{k\in\NN}
\Ind\bigl\{X^{*}_{t_{2k-1}}\in G'\bigr\} <\infty\right) =0\qquad
\text{and}\qquad
\Prob^{v^{*}}_{\imeas_{v^{*}}}\left( \sum_{k\in\NN}
\Ind\bigl\{X^{*}_{t_{2k}}\in G''\bigr\} <\infty\right) =0
\end{equation*}
Therefore,
\begin{equation}\label{e-mart05}
\Prob^{v^{*}}_{\imeas_{v^{*}}}\left( \sum_{k\in\NN}
\Ind\bigl\{X^{*}_{t_{2k-1}}\in G'\bigr\} =\infty\,,~
\sum_{k\in\NN}\Ind\bigl\{X^{*}_{t_{2k}}\in G''\bigr\} =\infty\right) =1\,.
\end{equation}
Therefore if $\nu$ is the number of upcrossings of $[a',b']$ by $M$
then \eqref{e-mart03} and \eqref{e-mart05} imply that
$\Prob^{v^{*}}_{\imeas_{v^{*}}}(\nu<\infty)=0$.
However by the reverse submartingale upcrossings inequality
$\Exp_{x}^{v^{*}}[\nu]<\infty$ which gives a contradiction
and the proof is complete.

\section{Conclusions}\label{S-concl}
We have proposed a nonlinear parabolic PDE that serves as a continuous time,
continuous state space analog of the relative value iteration scheme
for solving the ergodic dynamic programming equation in finite state problems.
This was done under a uniform stability condition in terms of an associated
Lyapunov function.

These results suggest several future directions:
\begin{remunerate}
\item
An important class of ergodic control problems is one wherein instability
is possible, but is heavily penalized by using a `near-monotone'
(see \cite[Chapter~3]{book} for a definition) running cost.
It would be both interesting and important to extend the above results
to this case as it covers several important applications.
\item
While the foregoing seems to extend easily to two-person zero-sum stochastic
differential games with ergodic payoffs, it would be of great interest
to do the same for interesting classes of non-cooperative games with ergodic
payoffs.
\item
Rate of convergence results, computational aspects, and convergence under
subgeometric ergodicity are also open issues.
\end{remunerate}

\def\cprime{$'$} \def\cprime{$'$} \def\cprime{$'$}
\providecommand{\bysame}{\leavevmode\hbox to3em{\hrulefill}\thinspace}
\providecommand{\MR}{\relax\ifhmode\unskip\space\fi MR }
\providecommand{\MRhref}[2]{%
  \href{http://www.ams.org/mathscinet-getitem?mr=#1}{#2}
}
\providecommand{\href}[2]{#2}

\end{document}